\documentclass[11pt, oneside]{article}   	% use "amsart" instead of "article" for AMSLaTeX format

\usepackage[a4paper, left=2.5cm, right=2.5cm, top=2.5cm]{geometry}%\geometry{landscape}                		% Activate for for rotated page geometry
\usepackage{graphicx}				% Use pdf, png, jpg, or epsÂ§ with pdflatex; use eps in DVI mode
\usepackage{caption}

\usepackage{subcaption}							% TeX will automatically convert eps --> pdf in pdflatex
\usepackage{pdfpages}
\usepackage{amssymb}
\usepackage{amsmath}
\usepackage{amsthm}
\usepackage[english]{babel}
\usepackage[utf8]{inputenc}
\usepackage[T1]{fontenc}
\usepackage{lmodern}
\usepackage{float}
\usepackage{framed, color}
\usepackage{enumitem}
\usepackage{abstract}
\usepackage[normalem]{ulem} % sout
\usepackage[colorlinks=true]{hyperref}

\newtheorem{theorem}{Theorem}[section]
\newtheorem{corollary}[theorem]{Corollary}
\newtheorem{lemma}[theorem]{Lemma}
\newtheorem{proposition}[theorem]{Proposition}
\newtheorem{example}[theorem]{Example}
\theoremstyle{definition}

\newtheorem{algorithm}[theorem]{Algorithm}

\newtheorem{assumption}{Assumption}

\newcommand{\R}{\mathbb{R}}
\newcommand{\N}{\mathbb{N}}

\DeclareMathOperator{\proj}{proj}

\DeclareMathOperator{\Div}{div}
\def\dx{\,\mathrm{d}x}
\def\dt{\,\mathrm{d}t}

\numberwithin{equation}{section}

\newcommand\BVSN[1]{\lvert #1 \rvert_{BV(\Omega)}}

\title{A penalty scheme to solve constrained non-convex optimization problems in $BV(\Omega)$}

\author{Carolin Natemeyer, Daniel Wachsmuth
\footnote{Institut f\"ur Mathematik,
Universit\"at W\"urzburg,
97074 W\"urzburg, Germany, {\tt carolin.natemeyer@mathematik.uni-wuerzburg.de, daniel.wachsmuth@mathematik.uni-wuerzburg.de}.
This research was partially supported by the German Research Foundation DFG under project grant Wa 3626/3-2.}}

\begin{document}
	\maketitle

{\bfseries Abstract.} We investigate non-convex optimization problems in 	$BV(\Omega)$ with two-sided pointwise inequality constraints.
We propose a regularization and penalization method to numerically solve the problem.
Under certain conditions, weak limit points of iterates are stationary for the original problem.
In addition, we prove optimality conditions for the original problem that contain Lagrange multipliers to the
inequality constraints.
Numerical experiments confirm the theoretical findings.



\section{Introduction}

Let $\Omega\in\R^n$ be an open bounded set
with Lipschitz boundary.
We consider the possibly non-convex optimization problem of the form
\begin{equation} \label{P_BV}
\min\limits_{u\in U_{ad} \cap BV(\Omega)}f(u)+ \BVSN{u}.
\end{equation}
Mostly, we will work with
\begin{equation} \label{eq:def_Uad}
U_{ad}=\{u\in BV(\Omega):\:u_a \le u(x)\le u_b \:\text{ f.a.a. } x\in \Omega\},
\end{equation}
where $u_a,u_b \in \R$.
The function space setting is $BV(\Omega)$, i.e., the space of functions of bounded variation that consists of $L^1(\Omega)$-functions with weak derivative in the Banach space $\mathcal{M}(\Omega)$ of real Borel measures on $\Omega$.
The term $\BVSN{u}$ denotes the $BV(\Omega)$-seminorm, which is equal to the total variation of the measure $\nabla u$, i.e., $\BVSN{u}=|\nabla u|(\Omega)$.
The functional $f:L^2( \Omega)\to\R$ is assumed to be smooth and can be non-convex. In particular, we have  in mind to choose $f(u):=f(y(u))$ as the reduced smooth part of an optimal control problem, incorporating the state equation.
We will give more details on the assumptions on the optimal control problem in Section \ref{sec:2}.
Problem \eqref{P_BV} is solvable, and existence of solutions to \eqref{P_BV} can be shown by the direct method of the calculus of variations, see Theorem \ref{existece_sol0}.
% We note this is not the case if we instead look for controls in the Hilbert space  $H^1(\Omega)$
% and consider
%  \[
%  \min\limits_{H^1(\Omega)} f(u)+\|\nabla u\|_{L^1(\Omega)}\quad\text{ s.t. } u\in U_{ad.}
%  \]

The purpose of the paper is two-fold:
\begin{enumerate}[label=(\Alph{enumi})]
  \item \label{purpose_A}
 We prove optimality conditions for \eqref{P_BV}--\eqref{eq:def_Uad} that contain Lagrange multipliers to the inequality constraints $u_a \le u$ and $u\le u_b$.
  Moreover, these multipliers belong to $L^2(\Omega)$.

 \item We investigate an algorithmic scheme to solve \eqref{P_BV}--\eqref{eq:def_Uad}, where weak limit points of iterates satisfy the optimality conditions from \ref{purpose_A}.
\end{enumerate}

Both of these goals rely on the same approximation method.
The algorithmic scheme consists of the following two parts.
First, we approximate the non-differentiable total variation term by a smooth approximation and apply a continuation strategy.
Second, we address the box constraints with a classical penalty method.
%This results in obtaining multiplier approximations that measure the constraints violation.
Of course, solutions to \eqref{P_BV} and appearing subproblems are not unique due to the lack of convexity, which makes the analysis  challenging.
In general, only stationary points of these non-convex problems can be computed.
Under suitable assumptions,
% In addition, stationary points of the arising subproblems do not converge to a global or local solution of \eqref{P_BV} in general.
% Under a certain boundedness condition on the iterates, we show that
limit points of the generated sequences of stationary points of the subproblems and of the associated multipliers satisfy a certain necessary optimality condition for the original problem.

In addition, we apply this regularization and penalization approach to local minima of the original problem.
This allows us to prove optimality conditions that contains Lagrange multipliers to the inequality constraints, see Theorem \ref{thm:final_local}.
Such a result is not available in the literature.
Admittedly, we have to make the assumption that the constraints $u_a$ and $u_b$ are constant functions.

Regularization by total variation is nowadays a standard tool in image analysis.
Following the seminal contribution \cite{TV_rudin_1992}, much research was devoted to study
such kind of optimization problems. We refer to \cite{chambolle_TV,overview_TV_2006}
for a general introduction and an overview on total variation in image analysis.
Optimal control problems in $BV$-spaces were studied for instance in \cite{CasasKunisch99_BV,CasasKunisch_ocsl_BV,Casas_TV_papablic17}.
These control problems are subject to semilinear equations, which results in non-convex control problems.
Finite element discretization and convergence analysis for optimization problems in $BV(\Omega)$ were investigated for instance in \cite{CasasKunisch99_BV,Bartels_TV_12,ClasonKunisch11}.
An extensive comparisons of algorithms to solve \eqref{P_BV} with the choice $f(u):=\frac12\|u-g\|_{L^2(\Omega)}$ can be found in \cite{Milicevic17}, see also \cite{MilicevicDiss}.
In \cite{SchevenSchmidt2018}, the one-sided obstacle problem in $BV(\Omega)$ is analyzed under low regularity requirements on the obstacle.
Interestingly, we could not find any results, where the existence of Lagrange multipliers to the inequality constraints in $BV(\Omega)$ is addressed.

One natural idea to regularize \eqref{P_BV} is to replace the non-differentiable $BV$-seminorm by a smooth approximation.
This was introduced in the image processing setting in \cite{Acar94_BV_approx} with the functional
\[
u\mapsto\int_\Omega\sqrt{\epsilon+|\nabla u|^2}\dx,
\]
which is widely used in the literature. Our regularization method is similar, with the exception that our functional guarantees existence of solutions in $H^1(\Omega)$.
A similar scheme was employed in the recent work \cite{hafemeyer2020}, where a path-following inexact Newton method for a convex PDE-constrained optimal control problem in $BV(\Omega)$ is studied.

	Let us comment on the structure of this work.
	In Section \ref{sec:2} we give a brief introduction to the function space $BV(\Omega)$ and recall some useful facts.
	Furthermore, we prove existence of solutions and a necessary first-order optimality condition for the optimization problem \eqref{P_BV}.
	In Section \ref{sec3}, we introduce the regularization scheme for \eqref{P_BV} and show that limit points of the suggested smoothing and penalty method
	satisfy a stationary condition for the original problem, see Theorem \ref{thm:final}.
	In Section \ref{sec4},
	we apply the regularization scheme to derive an optimality condition for locally optimal solutions of \eqref{P_BV}, see Theorem \ref{thm:final_local}.
	These conditions are	stronger than
	the conditions proven in Section \ref{sec:2}, since they
	contain Lagrange multipliers to the inequality constraints.
	Finally, we provide numerical results and details regarding the implementation of the method in Section \ref{sec5}.

\section{Preliminaries and Background} \label{sec:2}
In this section we want to  provide some definitions and results regarding the mathematical background of the paper.  For details we refer also to,  e.g.,
\cite{Attouch2006,ClasonKunisch11, CasasKunisch_ocsl_BV,CasasKunisch99_BV}.
First, let us  recall that $\mathcal{M}(\Omega)$ is the dual space of $C_0(\Omega)$. The noram of a measure $\mu\in \mathcal{M}(\Omega)$ is given by
\[
\|\mu\|_{\mathcal{M}(\Omega)}=\sup\left\{\int_\Omega zd\mu\: :\: z\in C_0(\Omega),\ |z(x)|\le1\:\ \forall x\in\Omega\right\}.
\]
The space of functions of bounded variation  $BV(\Omega)$ is a non-reflexive Banach space when endowed with the norm
\[
\|u\|_{BV(\Omega)}=\|u\|_{L^1(\Omega)}+\BVSN{u},
\]
where we define the total variation of $\nabla u$ by
\begin{equation}\label{def:total_var}
\BVSN{u}:=\sup\left\{\int_\Omega u\Div\varphi\dx:\:\varphi\in C^\infty_0(\Omega)^n,\:\ |\varphi(x)|\leq1\ \forall x\in\Omega \right\} =\|\nabla u\|_{\mathcal M(\Omega)^n}.
\end{equation}
Here, $|\cdot|$ denotes the Euclidean norm on $\R^n$.
In the definition \eqref{def:total_var}, $\nabla:BV(\Omega)\to\mathcal{M}(\Omega)^n$ is a linear and continuous map.
Functions in $BV(\Omega)$  are not necessarily continuous, as an example we mention the characteristic functions of a set with sufficient regularity.
If  $u\in W^{1,1}(\Omega)$, then $\|u\|_{BV(\Omega)}=\|u\|_{W^{1,1}(\Omega)}$ and $\BVSN{u}=\|\nabla u\|_{L^1(\Omega)}.$

The Banach space $BV(\Omega)$ and $\BVSN{\cdot}$ have some useful properties, which are recalled in the following.

\begin{proposition}\label{prop:BV}
Let $\Omega\subset \R^n$ be an open bounded set with Lipschitz boundary.
	\begin{enumerate}
	\item
The space  $BV(\Omega)$ is continuously embedded in $L^r(\Omega)$ for  $1\le r\le\frac{n}{n-1}$,  while the embedding
is compact for $1\le r<\frac{n}{n-1}$.
% 		\item Every bounded sequence $(u_k)\subset BV(\Omega)$ has a subsequence $(u_{k_n})$ such that there is $u\in BV(\Omega)$ and $u_{k_n}\to u$ strongly in $L^1(\Omega)$.
		\item Let $(u_k)\subset BV(\Omega)$
		be bounded in $BV(\Omega)$
		with $u_k\to u$ in $L^1(\Omega)$. Then
		\[
		\BVSN{u}\leq \liminf\limits_{k\to\infty}\BVSN{u_k}
		\]
		holds.
		\item
			%{\item\cite[Theorems 10.1.2.- 10.1.4.]{Attouch2006}}
			For $u\in BV(\Omega)\cap L^p(\Omega)$: $p\in[1,\infty)$, there is a sequence $(u_k)\subset C^\infty(\bar \Omega)$ such that
		\begin{equation} \label{eq_intermed_conv}
		u_k\to u \text{ in } L^p(\Omega) \text{ and } \BVSN{u_k}\to\BVSN{u},
		\end{equation}
		that is, $C^\infty(\bar \Omega)$ is dense in $BV(\Omega)\cap L^p(\Omega)$ with respect to the intermediate convergence \eqref{eq_intermed_conv}.
% 		\item The mapping $u\mapsto \BVSN{u}$ is convex.
	\end{enumerate}
\end{proposition}
\begin{proof}
(1) is \cite[Thm. 10.1.3, 10.1.4]{Attouch2006}.
(2) is \cite[Prop. 10.1.1(1)]{Attouch2006}.
 (3) Can be proven analogously to \cite[Theorem 10.1.2]{Attouch2006}, which contains the case $p=1$.
 A proof for $p\in (1,\infty)$ can be obtained by replacing $L^1$-norms by $L^p$-norms in the proof by \cite{Attouch2006},
 which is by a standard mollification procedure.
\end{proof}

\paragraph{Notation.}
Frequently, we will use the following standard notations from convex analysis.
The indicator function of a convex set $C$ is denoted by $\delta_C$.
The normal cone of a convex set $C$ at a point $x$ is denoted by
$N_C(x)$,  and $\partial h$ denotes the convex subdifferential of a convex function $h$.
It is well known that $\partial\delta_C(x)=N_C(x)$ holds for convex sets $C$.
Moreover, we introduce the notation
\[
J(u):= f(u)+\BVSN{u},
\]
which will be used thoughout the paper.
In addition, we will denote the positive and negative part of $x\in \R$ by $(x)_+:=\max(x,0)$ and $(x)_- := \min(x,0)$.

\subsection{Standing assumption}

%Let us collect assumptions on the  Problem \eqref{P_BV}.
In order to prove existence of solutions of  \eqref{P_BV} and to analyze the regularization scheme later on, we need some assumptions on the ingredients of the optimal control problem \eqref{P_BV}.
Let us start with collecting those in the following paragraph.
\begin{assumption} \label{ass:A}
\phantom{bla}
	\begin{enumerate}[label=(A\arabic{enumi})]
		\item \label{ass:A1} $f:L^2(\Omega)\to \R$  is bounded from below and weakly lower semicontinuous.
		\item \label{ass:A2} $f + \delta_{U_{ad}}$ is weakly coercive in $L^2(\Omega)$, i.e.,
		for all sequences $(u_k)$ with $u_k\in U_{ad}$ and $\|u_k\|_{L^2(\Omega)} \to \infty$ it follows $f(u_k)\to +\infty$.
		\item \label{ass:A3}
		$U_{ad}$ is a convex and closed subset of $L^2(\Omega)$ with $U_{ad} \cap BV(\Omega) \ne \emptyset$.

		\item \label{ass:A4} $f:L^2(\Omega)\to \R$ is continuously Fréchet differentiable.

		\item \label{ass:A5}
		$U_{ad}:=\{u\in L^2(\Omega):\:u_a\le u(x)\le u_b\:\text{ f.a.a. } x\in \Omega\}$ with $u_a,u_b\in \R$ and $u_a<u_b$.

	\end{enumerate}

\end{assumption}

Here, assumptions \ref{ass:A1}--\ref{ass:A3} will be used to prove existence of solutions of \eqref{P_BV}.
Condition \ref{ass:A4} is necessary to derive necessary optimality conditions.
The assumption \ref{ass:A5} will be used in Section \ref{sec3} to prove boundedness of Lagrange multipliers
associated to the inequality constraints in $U_{ad}$.

\begin{example}
We consider
\[
f(u):= \int_\Omega L(x, y_u(x))\dx,
\]
where $y_u\in H^1_0(\Omega)$ is defined as the unique weak solution to the elliptic partial differential equation
\[
(Ay)(x)+d(x,y(x)) = u(x)\quad \text{ a.e. in } \Omega.
\]
Let us assume that $A$ is an uniformly elliptic operator with bounded coefficients and $L,d$ are Carath{\'e}odory functions, continuously differentiable with respect to $y$ such that derivatives are bounded on bounded sets and that $d$ is monotonically increasing with respect to $y$. Then it is well known that $f$ is covered by Assumption \ref{ass:A}, see for instance \cite{CasasL12012}.

\end{example}
\begin{example}
Another example is given by the functional
\[
f(u):=\int_\Omega\int_I L(x,t,y_u(x,t))\dx\dt.
\]
with $y_u\in L^2(I,H^1_0(\Omega)),\: I:=(0,T),\:T>0,\: \Omega\subset\R^n,$
 as the solution of the parabolic equation
 \[
 \partial_ty(x,t) +(Ay)(x,t) +d(x,t,y(x,t)) = u(x,t) \quad\text{a.e. in }\Omega\times I.
 \]
 Assuming again an uniformly elliptic operator $A$ and measurable functions $L,d$ of class $C^2$ w.r.t. $y$ with bounded derivatives, such that $d$ is monotonically increasing, the functional $f$ satisfies Assumption \ref{ass:A}.
 We refer to \cite[Chapter 5]{BOOK_troeltzsch}, \cite{Casas_TV_papablic17}.

\end{example}

\subsection{Existence of solutions of \eqref{P_BV}}

Next, we show that under suitable assumptions on the function $f$, Problem \eqref{P_BV} has at least one solution in $L^2(\Omega)\cap BV(\Omega)$.

\begin{theorem}\label{existece_sol0}
	Let Assumptions \ref{ass:A1}--\ref{ass:A3} be satisfied. Then \eqref{P_BV} has a solution $u\in L^2(\Omega)\cap BV(\Omega)$.
\end{theorem}

\begin{proof}
	The proof is standard. We recall it by following  the lines of the proof of \cite[Theorem 2.1]{CasasKunisch99_BV}.
	Consider a minimizing sequence $(u_k)\subset L^2(\Omega)\cap BV(\Omega)$.
	Since $f$ is bounded from below by \ref{ass:A1},  $\left(\BVSN{u_k}\right)$ is bounded.
	By \ref{ass:A2}, $(u_k)$ is bounded in $L^2(\Omega)$, and hence $(u_k)$ is bounded in $BV(\Omega)$.
	By Proposition \ref{prop:BV}, there is a subsequence $(u_{k_n})$  and $\bar u\in L^2( \Omega)\cap BV(\Omega)$ with
	$u_{k_n}\rightharpoonup \bar u$ in $L^2(\Omega)$ and $u_{k_n}\to\bar u$ in $L^1(\Omega)$.
	Due to \ref{ass:A3}, $U_{ad}$ is weakly closed in $L^2(\Omega)$, and  $\bar u\in U_{ad}$ follows.
	By weak lower semicontinuity \ref{ass:A1} and Proposition \ref{prop:BV}, we obtain
	\[
	J(\bar u)\le\liminf\limits_{k_n\to\infty}J(u_{k_n}) = \inf\limits_{u\in L^2(\Omega)\cap BV(\Omega)}J(u),
	\]
	therefore, $\bar u$ is a solution.
\end{proof}

\subsection{Necessary optimality conditions}

Next, we provide  a first-order necessary optimality condition for \eqref{P_BV}. A similar result with proof can be found in \cite[Theorem 2.3]{CasasKunisch99_BV}.

\begin{theorem}\label{thm_FONC_BV}
	Let Assumptions \ref{ass:A3}--\ref{ass:A4} be satisfied.
Let $\bar u\in BV(\Omega)\cap U_{ad}$
be locally optimal for \eqref{P_BV} with respect to $BV(\Omega) {\cap L^2(\Omega)}$, i.e., there is $r>0$ such that $J(\bar u) \le J(u)$ for all $u\in U_{ad}$
with $\|u-\bar u\|_{BV(\Omega)}+\|u-\bar u\|_{L^2(\Omega)} <r$.
Then there is $\lambda\in \partial\|\cdot\|_{\mathcal{M}{(\Omega)}}(\nabla \bar u) \subset (\mathcal{M}(\Omega)^n)^*$ such that
	\begin{equation}\label{nec_opt_0}
	-\nabla f(\bar u)\in -\Div\lambda + N_{U_{ad}}(\bar u) \text{ in } (BV(\Omega)\cap L^2(\Omega))^*,
	\end{equation}
	where $-\Div:(\mathcal{M}(\Omega)^n)^*\to (BV(\Omega)\cap L^2(\Omega))^*$ is the adjoint operator of $\nabla:BV(\Omega)\cap L^2(\Omega)\to\mathcal{M}(\Omega)^n$,
	and the normal cone $N_{U_{ad}}$ of $U_{ad}$ is determined with respect to $BV(\Omega)\cap L^2(\Omega)$:
	\[
	 N_{U_{ad}} = \{ \mu \in (BV(\Omega)\cap L^2(\Omega))^*: \ \langle \mu, u-\bar u \rangle \le 0 \quad \forall u \in U_{ad} \cap BV(\Omega)\}.
	\]
\end{theorem}
\begin{proof}
Let us define $X:=BV(\Omega)\cap L^2(\Omega)$.
	By standard arguments, we find
	\[
	-\nabla f(\bar u)\in\partial(\BVSN{\cdot}+\delta_{U_{ad}})(\bar u) \subset X^*.
	\]
	Recall,  $\BVSN{u}=\|\nabla u\|_{\mathcal{M}(\Omega)}	 = (\|\cdot\|_{\mathcal{M}(\Omega)}\circ\nabla)(u)$.
	Following \cite[Theorem 2.3]{CasasKunisch99_BV}, let $-\Div:(\mathcal{M}(\Omega)^n)^*\to X^*$ denote the adjoint operator of $\nabla:X\to\mathcal{M}(\Omega)^n$, i.e.,
	\[
	\langle\Div \lambda,u\rangle=-\langle\lambda,\nabla u\rangle\quad \forall u\in X,\ \lambda \in (\mathcal{M}(\Omega)^n)^*.
	\]
	By the sum and chain rules for the convex subdifferential, \eqref{nec_opt_0} can be rewritten as
	\[
	-\nabla f(\bar u)\in -\Div \left( \partial\|\cdot\|_{\mathcal{M}{(\Omega)}}(\nabla \bar u)\right) + N_{U_{ad}}(\bar u),
	\]
	which is the claim.
\end{proof}

\section{Regularization scheme} \label{sec3}

In this section, we introduce the regularization scheme for \eqref{P_BV}.
We will use a smoothing of the $BV$-norm as well as a penalization of the constraints $u\in U_{ad}$.
In order to approximate the $BV$-norm by smooth functions,
we introduce the following family $(\psi_\epsilon)_{\epsilon>0}$ of smooth functions with $\psi_\epsilon(t) \approx |t|$.
For $\epsilon>0$ we define $\psi_{\epsilon}:\R^n\to\R^n$ by
\begin{equation} \label{eq:def_psi}
\psi_\epsilon(t):=\sqrt{\epsilon+|t|^2} +\epsilon |t|^2,
\end{equation}
where $|\cdot|$ denotes the Euclidian norm in $\R^{n}$.
As a first direct consequence of the above definition we have:
\begin{lemma}\label{ass:B}
For $\epsilon>0$,
 $(\psi_\epsilon)_{\epsilon>0}$ is a family of twice
continuously differentiable functions from $\R^n$ to $\R^n$
with the following properties:
\begin{enumerate}[label=(\arabic*), ref=Lemma \ref{ass:B}(\theenumi)]
	\item $t\mapsto\psi_\epsilon(t)$ is convex. \label{ass:B1}
	\item $\psi_\epsilon(t)\ge |t|+\epsilon |t|^2$ and $\psi'_\epsilon(t)t\ge0$ for all $t\in \R^n$. \label{ass:B2}
	\item For all $t\in\R^n$, $\epsilon\mapsto\psi_\epsilon(t)$ is monotonically increasing and $\psi_\epsilon(t)\to|t|$ as $\epsilon\searrow0.$\label{ass:B3}
	\item $t\mapsto\psi'_\epsilon(t)t$ is coercive, i.e., $\psi'_\epsilon(t)t\to\infty$ as $|t|\to \infty$.\label{ass:B4}
	\item $\psi'_\epsilon(t)t  \ge |t| - \sqrt\epsilon$  for all $t\in \R^n$.\label{ass:B5}
\end{enumerate}
\end{lemma}
\begin{proof}

 Properties (1)--(4) are immediate consequences of the definition.
 (5) can be proven as follows:
 \[
  \psi'_\epsilon(t)t -|t| \ge \frac{ |t|^2 - |t| \sqrt{\epsilon + |t|^2}  }{\sqrt{\epsilon + |t|^2}}
  = \frac{ -\epsilon |t|^2 }{ \sqrt{\epsilon + |t|^2} (|t|^2 + |t| \sqrt{\epsilon + |t|^2})}
  \ge \frac{ -\epsilon |t|^2 }{ \sqrt{\epsilon + |t|^2} |t|^2} \ge -\sqrt\epsilon.
 \]
\end{proof}
The $BV$-minimization problem \eqref{P_BV} is then approximated by
\begin{equation}
\label{P_eps}\tag{$P_\epsilon$}
\min\limits_{u\in H^1(\Omega)}f(u)+\int_\Omega\psi_\epsilon(\nabla u)\dx
\quad\text{ s.t. }u\in U_{ad}.
\end{equation}
The choice \eqref{eq:def_psi} of $\psi_\epsilon$ guarantees the existence of solutions of this problem in $H^1(\Omega)$.
Note that the standard approximation $|t| \approx \sqrt{ \epsilon + |t|^2}$ does \ref{ass:B2},
which ensures that $u\mapsto \int_\Omega\psi_\epsilon(\nabla u)\dx$ is weakly coercive in $H^1(\Omega)$.
The existence of solutions  $u\in H^1(\Omega)$ (and the fact that $\nabla u$ is a measurable function) is important for the subsequent analysis.

Due to the presence of the inequality constraints, Problem \eqref{P_eps} is difficult to solve.
Following existing approaches in the literature, see, e.g., \cite{Kunisch_Wa_2012,schiela_wachsmuth}, we will use a smooth penalization of these constraints.
We define the smooth function $\max_\rho$  by
\begin{equation}\label{eq:max_gamma}
{\max}_\rho(x):=\begin{cases}
\max(0,x)\quad&\text{ if } |x|\ge \frac{1}{2\rho},\\
\frac{\rho}{2}(x+\frac{1}{2\rho})^2&\text{ if }|x|\le\frac{1}{2\rho},
\end{cases}
\end{equation}
where $\rho>0$. Due to the inequalities
\begin{equation} \label{eq:est_max_rho}
0\le\max(0,x)\le{\max}_\rho(x)\le \max(x,0)+\frac{1}{2\rho}\quad \forall x\in\R,
\end{equation}
it can be considered as an approximation of $x\mapsto \max(x,0) = (x)_+$.
In addition, one verifies $\max(0,x)\le t^{-1}{\max}_\rho(tx)$ for all $x\in \R$ and $t>0$.
Let us introduce
\begin{equation}\label{def_Mrho}
M_\rho(x):=\int_{-\infty}^x {\max}_\rho(t)\dt.
\end{equation}
% \[
% M_\rho(x):=\begin{cases}
% \frac12\max(0,x)^2+\frac{1}{24\rho^2}\quad&\text{ if } |x|\ge \frac{1}{2\rho},\\
% \frac{\rho}{6}(x+\frac{1}{2\rho})^3&\text{ if }|x|\le\frac{1}{2\rho}.
% \end{cases}
% \]
%
Using this function, we define the penalized problem by
\begin{equation}\label{Peps_penal}\tag{$P_{\epsilon,\rho}$}
\min_{u\in H^1(\Omega)}  f(u)+\int_\Omega\psi_\epsilon(\nabla u)\dx+
\int_\Omega\frac 1{\rho}\left(M_\rho(\rho(u_a-u))+M_\rho(\rho( u-u_b))\right)\dx.
\end{equation}
If $u\in H^1(\Omega)$ is a local solution of \eqref{Peps_penal} then it satisfies
\begin{equation}\label{OC_eps,c}
\int_\Omega\nabla f(u)v+\psi_\epsilon'(\nabla u)\nabla v
-\lambda^a_{\rho}(u)v
+\lambda_{\rho}^b(u)v \dx=0 \quad \forall\: v\in H^1(\Omega),
\end{equation}
where we used the abbreviations
\begin{equation}\label{eq_def_lambda}
\lambda^a_{\rho}(u):={\max}_\rho(\rho(u_a-u)),\:\lambda^b_{\rho}(u):={\max}_\rho(\rho(u-u_b)).
\end{equation}
Existence of solutions of \eqref{Peps_penal} and necessity of \eqref{OC_eps,c} for local optimality can be proven by standard arguments.

\begin{corollary}\label{cor_OC_Eps_exist_sol}
 Let assumptions \ref{ass:A1}--\ref{ass:A5} be satisfied. Then the equation \eqref{OC_eps,c}
 admits a solution $u\in H^1(\Omega)$.
\end{corollary}

% \subsection{Convergence analysis of the classical penalty method}

We will investigate the behavior of a penalty and smoothing method to solve \eqref{P_BV}.
Since \eqref{Peps_penal} is a non-convex problem, it is unrealistic to assume that one can compute global solutions.
Instead, the iterates will be chosen as stationary points of \eqref{Peps_penal}.
Hence, we are interested in the behavior of stationary points $u_{\epsilon,\rho}$ and corresponding multipliers $\lambda_{\rho}^a(u_{\epsilon,\rho})$, $\lambda^b_{\rho}(u_{\epsilon,\rho})$ as $\rho\to\infty$
and $\epsilon\searrow0$.

The resulting method then reads as follows.
\begin{algorithm} \label{alg1}
Choose $\epsilon_0\in(0,1),\rho_0>0$ and $u_0\in H^1(\Omega)$.
\begin{enumerate}
\item
Compute $u_k$ as solution to
\begin{equation}\label{alg1:eq_k}
\int_\Omega\nabla f(u)v+\psi_{\epsilon_k}'(\nabla u)\nabla v
-\lambda^a_{\rho_k}(u)v
+\lambda_{\rho_k}^b(u)v \dx=0 \quad \forall\: v\in H^1(\Omega),
\end{equation}
where $\lambda^a_{\rho}(u),\lambda_{\rho}^b(u)$ are defined in \eqref{eq_def_lambda}.
\item Choose $\rho_{k+1}>\rho_k,\: \epsilon_{k+1}<\epsilon_k.$
\item If a suitable stopping criterion is satisfied: Stop.  Else set $k:=k+1$ and go to Step 1.
\end{enumerate}
\end{algorithm}

In view of Corollary \ref{cor_OC_Eps_exist_sol}, the algorithm is well-defined.
In the following, we assume that the algorithm generates an infinite sequence of iterates $(u_k, \lambda^a_{\rho_k}(u_k),\lambda^b_{\rho_k}(u_k))$.
Here, we are interested to prove that weak limit points are stationary, i.e., they satisfy the optimality condition \eqref{nec_opt_0} for \eqref{P_BV}.
Throughout the subsequent analysis, we assume that assumptions \ref{ass:A1}--\ref{ass:A5} are satisfied
\subsection{A-priori bounds} \label{sec:a-priori}

In order to investigate the sequences of iterates $(u_k)$ and its (weak) limit points, it is reasonable to  derive bounds of iterates $u_k$, i.e., solutions of \eqref{OC_eps,c}, first.
To this end, we will study solutions $u\in H^1(\Omega)$ of the nonlinear variational equation
\begin{equation}\label{eq_u_g}
\int_\Omega \psi_\epsilon'(\nabla u)\nabla v
-\lambda^a_{\rho}(u)v
+\lambda_{\rho}^b(u)v \dx=
\int_\Omega gv\dx
\quad \forall\: v\in H^1(\Omega)
\end{equation}
for $\epsilon>0$, $\rho>0$, and $g\in L^2(\Omega)$.
The functions $\lambda^a_\rho(u)$, $\lambda^b_\rho(u)$ are defined in \eqref{eq_def_lambda} as
\[
\lambda^a_{\rho}(u):={\max}_\rho(\rho(u_a-u)),\:\lambda^b_{\rho}(u):={\max}_\rho(\rho(u-u_b)).
\]

Let us start with the following lemma. It shows that the supports of the multipliers $\lambda^a_\rho(u)$, $\lambda^b_\rho(u)$
do not overlap if the penalty parameter is large enough.
\begin{lemma}\label{lem_lalb_disjoint}
Let $u\in H^1(\Omega)$ be a solution of \eqref{eq_u_g} to $g\in L^2(\Omega)$.
Suppose $\rho^2\ge\frac1{u_b-u_a}$. Then it holds
\[
 \lambda^a_\rho(u) \cdot \lambda^b_\rho(u) =0 \text{ a.e.\ on } \Omega.
 \]
\end{lemma}
\begin{proof}
 Let $x\in \Omega$ such that $\lambda^a_\rho(u)(x) \cdot \lambda^b_\rho(u)(x) \ne 0$. Then it holds
 $u(x) < u_a(x) + \frac1{2\rho^2}$ and  $u(x) > u_b(x) - \frac1{2\rho^2}$.
 Consequently, $\rho^2< \frac1{u_b-u_a}$ follows.
\end{proof}
Under the assumptions that the bounds $u_a$ and $u_b$ are constant, we
can prove the following series of helpful results regarding the boundedness of iterates.
 Let us start with the boundedness of the multiplier sequences.
This result is inspired by related results for the $H^1$-obstacle problem,
see, e.g., \cite[Lemma 5.1]{KinderlehrerStampacchia1980}, see also \cite[Lemma 2.3]{schiela_wachsmuth}.
The results for $H^1$-obstacle problems require the assumption $\Delta u_a, \Delta u_b\in L^2(\Omega)$.
It is not clear to us, how the following proof can be generalized to non-constant obstacles $u_a,u_b$.
\begin{lemma}\label{lem_lambda_bounded}
Let $u\in H^1(\Omega)$ be a solution of \eqref{eq_u_g} to $g\in L^2(\Omega)$.
Suppose $\rho^2\ge\frac1{u_b-u_a}$. Then it holds
 \[
   \|\lambda^a_\rho(u)\|_{L^2(\Omega)} +   \|\lambda^b_\rho(u)\|_{L^2(\Omega)}  \le 2\|g\|_{L^2(\Omega)}.
 \]
\end{lemma}
\begin{proof}
	To show boundedness of the multipliers $\lambda^a_\rho(u),\lambda^b_\rho(u)$ in $L^2(\Omega)$, we test the optimality condition \eqref{eq_u_g} with $\lambda^a_\rho(u)$ and $\lambda^b_\rho(u)$, respectively.
	We get for $\lambda^a_\rho(u)={\max}_{\rho}(\rho(u_a-u))$
	\[
	\|\lambda^a_\rho(u)\|^2_{L^2(\Omega)} = \int_\Omega \psi_{\epsilon}'(\nabla u) \nabla({\max}_{ \rho}(\rho(u_a-u)))\dx
	 - \int_\Omega g\lambda^a_\rho(u)\dx
	 + \int_\Omega\lambda^b_\rho(u)\lambda^a_\rho(u)\dx.
	\]
	Due to Lemma \ref{lem_lalb_disjoint}, the last term is zero.
	It remains to analyze the first term.
	Here, we find
	\[
	 \int_\Omega \psi_{\epsilon}'(\nabla u) \nabla {\max}_{ \rho}(\rho(u_a-u)) \dx
	 =\int_\Omega \psi_{\epsilon}'(\nabla u) \rho {\max}_{ \rho}'(\rho(u_a-u)) \nabla (-u) \dx \le0,
	\]
	where we used \ref{ass:B2} and ${\max}_{ \rho}' \ge0$.
	This proves $\|\lambda^a_\rho(u)\|_{L^2(\Omega)} \le \|g\|_{L^2(\Omega)}$.
	Similarly, $\|\lambda^b_\rho(u)\|_{L^2(\Omega)} \le \|g\|_{L^2(\Omega)}$ can be proven.
\end{proof}

\begin{corollary} \label{cor_constr_vio}
Let $u\in H^1(\Omega)$ be a solution of \eqref{eq_u_g} to $g\in L^2(\Omega)$.
Suppose $\rho^2\ge\frac1{u_b-u_a}$.
Then it holds
	\[
	\|(u-u_b)_+\|_{L^2(\Omega)}+\|(u_a-u)_+\|_{L^2(\Omega)} \le 2\rho^{-1}\|g\|_{L^2(\Omega)}.
	\]
\end{corollary}
\begin{proof}
 Due to the definition of $\max_\rho$, we have $\max(x,0) \le \rho^{-1}\max_\rho(\rho x)$ for all $x\in \R$.
 This implies
 \[
  \|(u-u_b)_+\|_{L^2(\Omega)} \le \rho^{-1} \|{\max}_{\rho}(\rho(u-u_b))\|_{L^2(\Omega)} = \rho^{-1}\|\lambda^b_\rho(u)\|_{L^2(\Omega)},
 \]
 and the claim follows by Lemma \ref{lem_lambda_bounded} above.
\end{proof}

\begin{corollary}\label{cor_bound_L2}
Let $u\in H^1(\Omega)$ be a solution of \eqref{eq_u_g} to $g\in L^2(\Omega)$.
Suppose $\rho^2\ge\frac1{u_b-u_a}$.
Then it holds
	\[
	\|u\|_{L^2(\Omega)} \le 2\rho^{-1}\|g\|_{L^2(\Omega)} +
	\| \max(|u_a|,|u_b|)\|_{L^2(\Omega)}.
	\]
\end{corollary}
\begin{proof}
 The claim is a consequence of Corollary \ref{cor_constr_vio} and the identity
\begin{equation}\label{eq35}
u= ( u-u_b)_+-(u_a- u)_++\proj_{[u_a,u_b]}( u).
\end{equation}
 \end{proof}

\begin{lemma}\label{lem_bound_W11}
Let $u\in H^1(\Omega)$ be a solution of \eqref{eq_u_g} to $g\in L^2(\Omega)$.
Suppose $\rho^2\ge\frac1{u_b-u_a}$.
 Then it holds
 \[
  \|\nabla u\|_{L^1(\Omega)} \le 3 \|g\|_{L^2(\Omega)}  \|u\|_{L^2(\Omega)}+\sqrt{\epsilon} |\Omega|.
 \]
\end{lemma}
\begin{proof}
We test the optimality condition \eqref{OC_eps,c} with $u$ and use \ref{ass:B5} to get the estimate
\[
 \begin{split}
\int_\Omega |\nabla u| -\sqrt{\epsilon} \dx\le
\int_\Omega \psi'_{\epsilon}(\nabla u)\nabla u\dx
& = \left|\int_\Omega gu-  \lambda^a_\rho(u) u + \lambda^b_\rho(u) u\dx\right| \\
& \le (\|g\|_{L^2(\Omega)} + \|\lambda^a_\rho(u)\|_{L^2(\Omega)}+ \|\lambda^b_\rho(u)\|_{L^2(\Omega)} ) \|u\|_{L^2(\Omega)}.
\end{split}
\]
The claim follows with the estimate of Lemma \ref{lem_lambda_bounded}.
\end{proof}

\subsection{Preliminary convergence results} \label{sec_prelim}

As next step, we derive convergence properties of solutions $u_k\in H^1(\Omega)$ of
\begin{equation}\label{eq_uk_gk}
\int_\Omega \psi_{\epsilon_k}'(\nabla u_k)\nabla v
-\lambda^a_{\rho_k}(u_k)v
+\lambda_{\rho_k}^b(u_k)v \dx=
\int_\Omega g_kv\dx
\quad \forall\: v\in H^1(\Omega)
\end{equation}
where
\begin{equation}\label{eq_ass_conv}
 \epsilon_k \searrow0, \ \rho_k \to +\infty, \ g_k \rightharpoonup g \text{ in } L^2(\Omega).
\end{equation}
From the results of the previous section, we immediately obtain that $(u_k)$ is bounded in $BV(\Omega)$, and $(\lambda^a_{\rho_k}(u_k))$ and $(\lambda^b_{\rho_k}(u_k))$
are bounded in $L^2(\Omega)$.
Moreover, strong limit points of $(u_k)$ satisfy the inequality constraints in \eqref{eq:def_Uad} due to Corollary \ref{cor_constr_vio}.
In a first result, we need to lift the strong convergence of (a subsequence of) $(u_k)$ in $L^1(\Omega)$,
which is a consequence of the compact embedding $BV(\Omega) \hookrightarrow L^r(\Omega)$, $r<\frac n{n-1}$,
to strong convergence in $L^2(\Omega)$.

\begin{lemma}\label{lem_strong_conv_L2}
Assume \eqref{eq_ass_conv}.
Let $u_k$, $k\in \N$, be a solution of \eqref{eq_uk_gk}.
Suppose $u_k \to u$ in $L^1(\Omega)$.
Then $u_k \to u$ in $L^2(\Omega)$.
\end{lemma}
\begin{proof}
 We use again the identity \eqref{eq35}:
 \[
u_k= ( u_k-u_b)_+-(u_a- u_k)_++\proj_{[u_a,u_b]}( u_k).
 \]
 By Corollary \ref{cor_constr_vio}, the first two terms converge to zero in $L^2(\Omega)$,
 which implies $u_a \le u\le u_b$ almost everywhere.
 The sequence $(\proj_{[u_a,u_b]}( u_k))$ converges to $\proj_{[u_a,u_b]}(u)=u$ in $L^1(\Omega)$
 and is bounded in $L^\infty(\Omega)$. By H\"older inequality it converges in $L^2(\Omega)$.
 \end{proof}

We are now in the position to prove existence of suitably converging subsequence under assumption \eqref{eq_ass_conv}.

\begin{theorem}\label{thm310}
Assume \eqref{eq_ass_conv}.
Let $(u_k)$, $k\in \N$, be a family of solutions of \eqref{eq_uk_gk}.
Then there is a subsequence such that
$u_{k_n} \to u^*$, $\lambda^a_{\rho_{k_n}}(u_{k_n})\rightharpoonup \lambda^a$, and $\lambda^b_{\rho_{k_n}}(u_{k_n}) \rightharpoonup \lambda^b$ in $L^2(\Omega)$.
\end{theorem}
\begin{proof}
 By Lemmas \ref{lem_lambda_bounded}, Corollary \ref{cor_bound_L2},  and Lemma \ref{lem_bound_W11},  $(u_k)$ is bounded in $BV(\Omega){\cap L^2(\Omega)}$, and $(\lambda^a_{\rho_k}(u_k))$ and $(\lambda^b_{\rho_k}(u_k))$
are bounded in $L^2(\Omega)$.
Then we can choose $(u_{k})$ as a subsequence that converges strongly in $L^1(\Omega)$ by Proposition \ref{prop:BV}.
Due to Lemma \ref{lem_strong_conv_L2} this convergence is strong in $L^2(\Omega)$.
Now, extracting additional weakly converging subsequences from $(\lambda^a_{\rho_k}(u_k))$ and $(\lambda^b_{\rho_k}(u_k))$ finishes the proof.
\end{proof}

The next result shows that limit points of $(u_k,\lambda^a_{\rho_k}(u_k),\lambda^b_{\rho_k}(u_k))$
satisfy the usual complementarity conditions.

\begin{lemma} \label{lem:complemtarity_cond_prelim}
Assume \eqref{eq_ass_conv}.
Let $u_k$, $k\in \N$, be a solution of \eqref{eq_uk_gk}.
% 	Assume $\rho_k\to\infty$.
% 	Suppose  $(\nabla f(u_k))$ is bounded in $L^2(\Omega)$.
	Let $u_k \to u^*$, $\lambda^a_{\rho_k}(u_k)\rightharpoonup \lambda^a$, and $\lambda^b_{\rho_k}(u_k) \rightharpoonup \lambda^b$ in $L^2(\Omega)$. Then it follows that
	\[
	(\lambda^b,\,u^*-u_b) = 0 \text{ and } (\lambda^a,\,u_a-u^*) = 0.
	\]
\end{lemma}

\begin{proof}
Due to  \eqref{eq_ass_conv},  $(g_k)$ is bounded in $L^2(\Omega)$.
	We deduce using \eqref{eq:est_max_rho}
	\begin{align*}
	|(\lambda^b_{\rho_k}(u_k),\,u_k-u_b)| &\le  \int_\Omega {\max}_{\rho_k}(\rho_k(u_k-u_b))|u_k-u_b|\dx \\
	&\le \int_\Omega\left(\rho_k(u_k-u_b)_++\frac1{2\rho_k}\right)|u_k-u_b|\dx\\
	& = \rho_k\int_\Omega(u_k-u_b)_+(u_k-u_b)\dx +\int_{\Omega}\frac{1}{2\rho_k}|u_k-u_b|\dx \\
	& = \rho_k \|(u_k-u_b)_+\|_{L^2(\Omega)}^2 +\frac{1}{2\rho_k}\|u_k-u_b\|_{L^1(\Omega)}.
	\end{align*}
Due to Corollary \ref{cor_constr_vio} and the boundedness of $(u_k)$ in $L^2(\Omega)$, both expressions tend to zero for $k\to\infty$.
This implies
	\[
	(\lambda^b,\,u^*-u_b)  =\lim_{k\to\infty}	(\lambda^b_{\rho_k}(u_k),u_k-u_b) =0.
	\]
	The same argumentation yields the claim for $(\lambda^a,\,u_a-u^*)$.
\end{proof}

We will now  show that weak limit points of solutions to \eqref{eq_uk_gk} satisfy a  stationary  condition similar to the one for the original problem \eqref{P_BV}.
We will utilize this result twice: first we apply it to iterates of Algorithm \ref{alg1}, second we will use it to prove a  optimality condition for \eqref{P_BV} that
has a different structure than that of Theorem \ref{thm_FONC_BV}.

\begin{theorem} \label{thm:final_prelim}
Assume \ref{ass:A1}--\ref{ass:A5} and \eqref{eq_ass_conv}.
Let $(u_k)$, $k\in \N$, be a family of solutions of \eqref{eq_uk_gk}.
	Let $u_k \to u^*$, $\lambda^a_{\rho_k}(u_k)\rightharpoonup \lambda^a$, and $\lambda^b_{\rho_k}(u_k) \rightharpoonup \lambda^b$ in $L^2(\Omega)$.
Then it holds
\begin{equation}\label{eq_osys_comp}
\begin{aligned}
 u^* & \in U_{ad} \\
 \lambda^a & \ge 0, & \lambda^b &\ge0, \\
  (\lambda^a,\,u_a-u^*)&= 0, &   (\lambda^b,\,u^*-u_b)&= 0.
 \end{aligned}
\end{equation}
In addition, there is  $\mu^*\in L^\infty(\Omega)^n$ with $\Div \mu^*\in L^2(\Omega)$  such that
	\[
	-\Div \mu^* -\lambda^a+\lambda^b =g
	\]
and
\[
 -\Div \mu^* \in \partial(\BVSN{\cdot})(u^*).
\]
Moreover, there is $\lambda^* \in \partial\|\cdot\|_{\mathcal{M}{(\Omega)}}(\nabla u^*) \subset (\mathcal{M}(\Omega)^n)^*$ with  $\Div \lambda^*\in L^2(\Omega)$ such that
	\[
	-\Div \lambda^* -\lambda^a+\lambda^b =g.
	\]
%
% 	Moreover, it holds
% 	\begin{equation}
% (-\Div\mu^*,u^*)=\BVSN{u^*} \text{ and } (-\Div\mu^*, v)\le\BVSN{v}
% 	\end{equation}
% 	for all $v\in BV(\Omega)\cap L^2(\Omega)$.
\end{theorem}

\begin{proof}
The system \eqref{eq_osys_comp} is a consequence of Corollary \ref{cor_constr_vio}, Lemma \ref{lem:complemtarity_cond_prelim},
and the non-negativity of $\max_\rho$.
In order to pass to the limit in \eqref{eq_uk_gk},
we need to analyze the term involving $\psi_{\epsilon_k}'(\nabla u_k)$.
Here, we argue similar as in the proof of \cite[Theorem 10]{CasasKunisch_ocsl_BV}.
Let $v\in C_c^\infty(\Omega)$.
Then, we find
\[
 \int_\Omega \psi_{\epsilon_k}'(\nabla u_k)\nabla v \dx
 =\int_\Omega \frac{\nabla u_k}{\sqrt{\epsilon_k+|\nabla u_k|^2}} \nabla v + 2\epsilon \nabla u_k \nabla v \dx
 = \int_\Omega \frac{\nabla u_k}{\sqrt{\epsilon_k+|\nabla u_k|^2}} \nabla v - 2\epsilon u_k \Delta v \dx.
\]
		Let us define $ \mu_k\in L^\infty(\Omega)$ by
		\[
		\mu_k:=\frac{\nabla u_k}{\sqrt{\epsilon_k+|\nabla u_k|^2}}.
		\]
		Clearly, the sequence $(\mu_k)$ is bounded in $L^\infty(\Omega)^n$,
		and there exists a subsequence converging weak-star in $L^\infty(\Omega)^n$.
		W.l.o.g.\@ we can assume $\mu_k\rightharpoonup^* \mu^* $ in $L^\infty(\Omega)^n$.
		Since $(u_k)$ is bounded in $L^2(\Omega)$, we obtain
		\[
		  \lim_{k\to\infty} \int_\Omega \psi_{\epsilon_k}'(\nabla u_k)\nabla v \dx
		  = \lim_{k\to\infty} \int_\Omega \mu_k \nabla v - 2\epsilon u_k \Delta v \dx
		  = \int_\Omega\mu^*\nabla v\dx.
		\]
Then we can pass to the limit in \eqref{eq_uk_gk} to find
\[
		\int_\Omega\mu^*\nabla v -\lambda^a v + \lambda^b v \dx= \int_\Omega gv \dx
\]
which is satisfied for all $v\in C_c^\infty(\Omega)$.
This implies
	\[
	-\Div \mu^* = g + \lambda^a - \lambda^b  \in L^2(\Omega).
	\]
	Let now  $v\in C^\infty(\bar \Omega)$. By convexity of $\psi_\epsilon$, we have
	\begin{equation}\label{eq_subgrad_ineq_psik}
	 \int_\Omega\psi_{\epsilon_k}(\nabla u_k) +  \psi'_{\epsilon_k}(\nabla u_k)(\nabla v-\nabla u_k)\dx \le \int_\Omega\psi_{\epsilon_k}(\nabla v)\dx.
	\end{equation}
	Here, we find $\int_\Omega\psi_{\epsilon_k}(\nabla v)\dx \to \|\nabla v\|_{L^1(\Omega)}$ and
	\[
	\liminf_{k\to\infty} \int_\Omega\psi_{\epsilon_k}(\nabla u_k) \dx \ge \liminf_{k\to\infty}   \|\nabla u_k\|_{L^1(\Omega)} \ge \BVSN{u^*},
	\]
	cf.\@, Proposition \ref{prop:BV}.
	Here, we used that $(u_k)$ is bounded in $BV(\Omega)$ due to Lemma \ref{lem_bound_W11} and \eqref{eq_ass_conv}.
	Using the equation \eqref{eq_uk_gk}, we find
	\[\begin{split}
	  \int_\Omega  \psi'_{\epsilon_k}(\nabla u_k)(\nabla v-\nabla u_k)\dx
	  & = \int_\Omega (g_k +\lambda^a_{\rho_k}(u_k) -\lambda_{\rho_k}^b(u_k))(v-u_k)\dx\\
 	  &\to \int_\Omega (g+\lambda^a -\lambda^b)(v-u^*) \dx
 	  = \int_\Omega -\Div \mu^* (v-u^*) \dx.
	\end{split}\]
	Then we can pass to the limit in \eqref{eq_subgrad_ineq_psik} to obtain
	\[
	 \BVSN{u^*} +  \int_\Omega -\Div \mu^* (v-u^*) \dx \le \BVSN{v}
	\]
	for all $v\in C^\infty(\bar \Omega)$.
	Due to the density result of Proposition \ref{prop:BV} with respect to intermediate convergence \eqref{eq_intermed_conv},
	the inequality holds for all $v\in BV(\Omega) \cap L^2(\Omega)$.
	Consequently,  $-\Div \mu^* \in \partial(\BVSN{\cdot})(u^*) \subset (BV(\Omega)\cap L^2(\Omega))^*$.
	Using the chain rule as in Theorem \ref{thm_FONC_BV}, we find
	\[
	 -\Div \mu^* \in -\Div \left( \partial\|\cdot\|_{\mathcal{M}{(\Omega)}}(\nabla u^*)\right),
	\]
	which proves the existence of $\lambda^*$ with the claimed properties.
\end{proof}

\subsection{Convergence of iterates}

We are now going to apply the results of the previous two sections to the iterates of  Algorithm \ref{alg1}.
In terms of \eqref{eq_uk_gk}, we have to set $g_k:=\nabla f(u_k)$.
As can be seen from, e.g., Theorem \ref{thm:final_prelim}, the boundedness of $(\nabla f(u_k))$ in $L^2(\Omega)$ will be crucial for any convergence analysis.
Unfortunately, this boundedness can only be guaranteed in exceptional cases.
Here, we prove it under the assumption that $\nabla f$ is globally Lipschitz continuous.
In Section \ref{sec:global} we show that convexity of $f$ or global optimality of $u_k$ is sufficient.

\begin{lemma}
Let $\nabla f: L^2(\Omega) \to L^2(\Omega)$ be globally Lipschitz continuous with modulus $L_f$.
Assume $\rho_k\to\infty$.
Then $(u_k)$ and $(\nabla f(u_k))$ are bounded in $L^2(\Omega)$.
\end{lemma}
\begin{proof}
Due to the Lipschitz continuity of $\nabla f$, we have
\[
 \|\nabla f(u_k)\|_{L^2(\Omega)} \le L_f  \|u_k\|_{L^2(\Omega)}  + \|\nabla f(0)\|_{L^2(\Omega)}.
\]
By Corollary \ref{cor_bound_L2}, we find for $k$ sufficiently large
\[
\begin{split}
	\|u_k\|_{L^2(\Omega)} & \le 2\rho_k^{-1}\|\nabla f(u_k)\|_{L^2(\Omega)} + \| \max(|u_a|,|u_b|)\|_{L^2(\Omega)}\\
	&\le  2\rho_k^{-1}L_f \|u_k\|_{L^2(\Omega)} + 2\rho_k^{-1}\|\nabla f(0)\|_{L^2(\Omega)}+ \| \max(|u_a|,|u_b|)\|_{L^2(\Omega)}.
	\end{split}
\]
If $k$ is such that $2\rho_k^{-1}L_f  < \frac12$, then
$
 \|u_k\|_{L^2(\Omega)} \le 4\rho_k^{-1}\|\nabla f(0)\|_{L^2(\Omega)}+ 2\| \max(|u_a|,|u_b|)\|_{L^2(\Omega)},
$
which proves the claim.
\end{proof}
The next observation is a simple consequence of previous results and shows the close relation between boundedness of $(u_k)$ in $L^2(\Omega)$ and $BV(\Omega)$  and the boundedness of $(\nabla f(u_k))$ in $L^2(\Omega)$.

\begin{lemma} \label{lem:L2conv}
	Assume $\rho_k\to\infty$.
	Then the following statements are equivalent:
\begin{enumerate}[label=(\arabic*)]
\item\label{lem314_1} 	$(\nabla f(u_k))$ is bounded in $L^2(\Omega)$,
\item\label{lem314_2} 	$(u_k)$ is bounded in $L^2(\Omega)$ and $BV(\Omega)$,
\item\label{lem314_3} $\{u_k: \ k\in \mathbb N\}$ is pre-compact in $L^2(\Omega)$.
\end{enumerate}

\end{lemma}
\begin{proof}

\ref{lem314_1} $\Rightarrow$ \ref{lem314_2}:
The boundedness of $(u_k)$
is a direct consequence of Corollary \ref{cor_bound_L2} and Lemma \ref{lem_bound_W11}.
\ref{lem314_2} $\Rightarrow$ \ref{lem314_3}
follows from Proposition \ref{prop:BV} and Lemma \ref{lem_strong_conv_L2}.
\ref{lem314_3} $\Rightarrow$ \ref{lem314_1}:
Since $u\mapsto \nabla f(u)$  is continuous from $L^2(\Omega)$ to $L^2(\Omega)$ by \ref{ass:A4}, the set $\{\nabla f(u_k): \ k\in \mathbb N\}$ is pre-compact in $L^2(\Omega)$
and thus bounded.
\end{proof}

Similarly to Theorem \ref{thm310}, we have the following result on the existence of converging subsequences.

\begin{theorem}
	Suppose $\epsilon_k\searrow0$ and $\rho_k\to\infty$.
	Let $(u_k)$ solve \eqref{OC_eps,c}.
	Assume that $(\nabla f(u_k))$ is bounded in $L^2(\Omega)$.
Then there is a subsequence such that
$u_{k_n} \to u^*$, $\lambda^a_{\rho_{k_n}}(u_{k_n})\rightharpoonup \lambda^a$, and $\lambda^b_{\rho_{k_n}}(u_{k_n}) \rightharpoonup \lambda^b$ in $L^2(\Omega)$.
\end{theorem}
\begin{proof}
 This result can be proven with similar arguments as  Theorem \ref{thm310}.
\end{proof}

We finally arrive at the following convergence result for iterates of Algorithm \ref{alg1} which is a consequence of Theorem \ref{thm:final_prelim}.

\begin{theorem} \label{thm:final}
Assume \ref{ass:A1}--\ref{ass:A5}.
	Suppose $\epsilon_k\searrow0$ and $\rho_k\to\infty$.
	Let $(u_k)$ solve \eqref{OC_eps,c}.
	Assume that there is a subsequence with $u_{k_n}\to u^*$, $\lambda_{k_n}^a \rightharpoonup\lambda^a$,  and $\lambda_{k_n}^b \rightharpoonup\lambda^b$ in $L^2(\Omega)$.
	Then it holds
\[%begin{equation}
\begin{aligned}
 u^* & \in U_{ad} \\
 \lambda^a & \ge 0, & \lambda^b &\ge0, \\
  (\lambda^a,\,u_a-u^*)&= 0, &   (\lambda^b,\,u^*-u_b)&= 0.
 \end{aligned}
\]%end{equation}
In addition, there is $\mu^*\in L^\infty(\Omega)^n$ with $\Div \mu^*\in L^2(\Omega)$  such that
	\[
	-\Div \mu^* -\lambda^a+\lambda^b = -\nabla f(u^*)
	\]
and
\[
 -\Div \mu^* \in \partial(\BVSN{\cdot})(u^*).
\]
Moreover, there is $\lambda^* \in \partial\|\cdot\|_{\mathcal{M}{(\Omega)}}(\nabla u^*) \subset (\mathcal{M}(\Omega)^n)^*$ with  $\Div \lambda^*\in L^2(\Omega)$ such that
	\[
	-\Div \lambda^* -\lambda^a+\lambda^b =g.
	\]
\end{theorem}
\begin{proof}
By assumption, we have $\nabla f(u_{k_n}) \to \nabla f(u^*)$.
The proof is now a direct consequence of Theorem \ref{thm:final_prelim}.
\end{proof}

\subsection{Global solutions}\label{sec:global}

The next theorem shows that global optimality is sufficient to obtain boundedness of iterates.
We note that if $f$ is convex, solutions of \eqref{OC_eps,c} are global solutions to the penalized problem
\eqref{Peps_penal}.

\begin{theorem} \label{lem:conv_case}
Assume \ref{ass:A1}--\ref{ass:A5}.
	Suppose $\epsilon_k\searrow0$ and $\rho_k\to\infty$.
	% and $\rho_k\to+\infty,\:\rho_{k+1}>\rho_k$.
	Suppose $(u_{k})$ is the corresponding sequence of global solutions to the penalized problems \eqref{Peps_penal}.
	Then $(u_k)$ is bounded in $BV(\Omega)\cap L^2(\Omega)$.% and $(\nabla f(u_k))$ is bounded in $L^2(\Omega)$.
\end{theorem}

\begin{proof}
	We introduce the notation
	\[
	j_{\epsilon,\rho}(u):= f(u)+\int_\Omega\psi_\epsilon(\nabla u)\dx+
	\int_\Omega\frac 1{\rho}\left(M_\rho(\rho(u_a-u))+M_\rho(\rho( u-u_b))\right)\dx.
	\
	\]
Set $\tilde u:= \frac12(u_a+u_b) \in H^1(\Omega)$. Then $j_{\epsilon_{k},\rho_k}(\tilde u)=f(\tilde u)$ for $\rho_k$ large enough.
Let $u_k$ be a global minimizer of $j_{\epsilon_{k},\rho_k}$.
This implies
\[%\begin{multline}
f(u_k)+\int_\Omega\psi_{\epsilon_{k}}(\nabla u_k)\dx
+ \int_\Omega\frac 1{{\rho_k}}\left(M_{\rho_k}({\rho_k}(u_a-u))+M_{\rho_k}({\rho_k}( u-u_b))\right)\dx
\le f(\tilde u).
\]%\end{multline}
Since $f$ is bounded from below, there is $K>0$ such that
\[
\int_\Omega\psi_{\epsilon_{k}}(\nabla u_k)\dx
+ \int_\Omega\frac 1{{\rho_k}}\left(M_{\rho_k}({\rho_k}(u_a-u))+M_{\rho_k}({\rho_k}( u-u_b))\right)\dx
\le K.
\]
This proves that $(\nabla u_k)$ is bounded in $L^1(\Omega)$ by \ref{ass:B2}.
By construction, we have
\[
 M_\rho(x) = \int_{-\infty}^x {\max}_\rho(t)\dt \ge  \int_{-\infty}^x \max(t,0)  \dt = \frac12 \max(0,x)^2.
\]
This implies
\[
 \frac {\rho_k}2 \left( \|(u_k-u_b)_+\|_{L^2(\Omega)}^2+\|(u_a-u)_+\|_{L^2(\Omega)} ^2 \right) \le K,
\]
and the boundedness of $(u_k)$ in $L^2(\Omega)$ is now a consequence of identity \eqref{eq35}.
\end{proof}
%----------------------------------------------------------------
\section{Optimality condition by regularization} \label{sec4}

Let us assume $\bar u\in BV(\Omega) \cap L^2(\Omega)$ is locally optimal to \eqref{P_BV}.
% %
% \begin{equation}\label{eq:P_loc}
% \min\limits_{u\in BV(\Omega)} f(u)+\BVSN{u}\quad \text{s.t. }u\in U_{ad}.
% \end{equation}
In this section, we want to show that
there is a sequence of solutions $(u_{\rho,\epsilon})$ of certain regularized problems converging to $\bar u$.
This will allow us to prove optimality conditions for $\bar u$ that are similar to the systems obtained in Theorems \ref{thm:final_prelim} and \ref{thm:final}.
Again, we work under the assumptions \ref{ass:A1}--\ref{ass:A5}.

The solution $\bar u$ satisfies the necessary optimality condition
\begin{equation}\label{eq:OC_lok}
-\nabla f(\bar u)\in \partial\left(\BVSN{\cdot}\right)(\bar u) +N_{U_{ad}}(\bar u) \text{ in } (BV(\Omega)\cap L^2(\Omega))^*,
\end{equation}
see also Theorem \ref{thm_FONC_BV}.
It is easy to see that \eqref{eq:OC_lok} implies that $\bar u$ is the unique solution to the linearized, strictly convex problem
\begin{equation}\label{eq:P_lin}
\min\limits_{u\in BV(\Omega)\cap L^2(\Omega)} \nabla f(\bar u) \cdot u+\BVSN{u}+\frac12\|u-\bar u\|_{L^2(\Omega)}^2 + \delta_{U_{ad}}(u).
\end{equation}

In fact, let $u^*\in BV(\Omega)\cap L^2(\Omega)$ be the solution of \eqref{eq:P_lin}. Then we have the following optimality condition
\[
-\nabla f(\bar u)\in \partial\left(\BVSN{\cdot}\right)(u^*) +(u^*-\bar u) + N_{U_{ad}}(u^*) \text{ in } (BV(\Omega)\cap L^2(\Omega))^*,
\]
which is satisfied by $u^*:=\bar u$.
Let us approximate \eqref{eq:P_lin} by the family of unconstrained convex problems
\begin{equation}\label{eq:P_lin_pen}
\min\limits_{u\in H^1(\Omega)}  \nabla f(\bar u)\cdot u+\int_\Omega\psi_\epsilon(\nabla u)\dx+\frac12\|u-\bar u\|_{L^2(\Omega)}^2 +\frac1\rho\int_\Omega M_\rho(\rho(u_a-u))+M_\rho(\rho(u-u_b))\dx.
\end{equation}
The optimality condition for the unique solution $u_{\epsilon,\rho}$ to \eqref{eq:P_lin_pen} is given by
\begin{equation} \label{eq:P_lin_pen_OC}
\int_\Omega \nabla f(\bar u)v +\psi'_\epsilon(\nabla u_{\epsilon,\rho})\nabla v+(u_{\epsilon,\rho}-\bar u)v
-\lambda^a_{\rho}(u_{\epsilon,\rho})v
+\lambda_{\rho}^b(u_{\epsilon,\rho})v \dx=0.
\end{equation}
for all $v\in H^1(\Omega)$.

\begin{corollary}\label{cor41}
	Suppose $\epsilon_k\searrow0$ and $\rho_k\to\infty$.
	Suppose $(u_{k})$ is the corresponding sequence of global solutions to the penalized problems \eqref{eq:P_lin_pen}.
	Then $(u_k)$ is bounded in $BV(\Omega)\cap L^2(\Omega)$.
\end{corollary}
\begin{proof}
The claim follows by a similar argumentation as in the proof of Theorem \ref{lem:conv_case}.
\end{proof}

\begin{lemma}\label{lem42}
	Suppose $\epsilon_k\searrow0$ and $\rho_k\to\infty$. Let $(u_{k})$ be the corresponding sequence of global solutions to the penalized problems \eqref{eq:P_lin_pen}.
	Then $u_k\to\bar u $ in $L^2(\Omega)$,
	and the sequences $(\lambda^a_{\rho_k}(u_k))$ and $(\lambda^b_{\rho_k}(u_k))$
	are bounded in $L^2(\Omega)$.
\end{lemma}

\begin{proof}
Due to Corollary \ref{cor41}, $(u_k)$ is bounded in $BV(\Omega)\cap L^2(\Omega)$.
Suppose for the moment $u_k \to u^*$ in $L^1(\Omega)$  and $u_k \rightharpoonup u^*$ in $L^2(\Omega)$.
By Lemma \ref{lem_strong_conv_L2} applied to $g_k:=-\nabla f(\bar u) - (u_k -\bar u)$ and $g:=-\nabla f(\bar u) - (u^* -\bar u)$, we obtain $u_k \to u^*$ in $L^2(\Omega)$.
By Theorem \ref{thm310}, the corresponding sequences $(\lambda^a_{\rho_k}(u_k))$ and $(\lambda^b_{\rho_k}(u_k))$
are bounded in $L^2(\Omega)$.
Suppose $\lambda^a_{\rho_k}(u_k) \rightharpoonup \lambda^a$ and $\lambda^b_{\rho_k}(u_k)\rightharpoonup\lambda^b$ in $L^2(\Omega)$.
By Theorem \ref{thm:final_prelim}, we have
\[
-\nabla f(\bar u) - (u^* -\bar u) + \lambda^a - \lambda^b\in \partial(\BVSN{\cdot})(u^*).
	\]
Due to the complementarity conditions \eqref{eq_osys_comp} of Theorem \ref{thm:final_prelim},
we get
\[
	-\nabla f(\bar u) - (u^* -\bar u)  \in  \partial\left(\BVSN{\cdot}+\delta_{U_{ad}}\right)(u^*).
\]
Since $-\nabla f(\bar u) \in \partial\left(\BVSN{\cdot}+\delta_{U_{ad}}\right)(\bar u)$, we have by the monotonicity of the subdifferential
\[
	(-\nabla f(\bar u) - (u^* -\bar u) - (-\nabla f(\bar u)), \ u^*-\bar u)\ge0,
\]
which implies $u^*=\bar u$.

With similar arguments, we can show that every subsequence of $(u_k)$ contains another subsequence that converges in $L^2(\Omega)$ to $\bar u$.
Hence, the convergence of the whole sequence follows.
\end{proof}

This convergence result enables us to prove that $\bar u$ satisfies an optimality condition
similar to those of Theorems \ref{thm:final_prelim} and \ref{thm:final}.

\begin{theorem} 	\label{thm:final_local}
Assume \ref{ass:A1}--\ref{ass:A5}.
Let $\bar u$ be  locally optimal for \eqref{P_BV}.
Then there is
\[
\lambda^* \in \partial\|\cdot\|_{\mathcal{M}{(\Omega)}}(\nabla \bar u) \subset (\mathcal{M}(\Omega)^n)^*
\]
with  $\Div \lambda^*\in L^2(\Omega)$
 such that
	\[
	-\Div \lambda^* -\lambda^a+\lambda^b = \nabla f(\bar u)
	\]
	and
	\[%begin{equation}
\begin{aligned}
 \lambda^a & \ge 0, & \lambda^b &\ge0, \\
  (\lambda^a,\,u_a-u^*)&= 0, &   (\lambda^b,\,u^*-u_b)&= 0.
 \end{aligned}
\]%end{equation}

\end{theorem}
\begin{proof}
We define $(u_k)$ as global solutions to the penalized problems \eqref{eq:P_lin_pen} to parameter sequences
	$\epsilon_k\searrow0$ and $\rho_k\to\infty$. Due to Lemma \ref{lem42}, we have $u_k \to\bar u$ in $L^2(\Omega)$.
	Define $g_k:=-\nabla f(\bar u) - (u_k-\bar u)$ and $g:=-\nabla f(\bar u)$.
	Now, the claim follows by Theorem \ref{thm:final_prelim}.
\end{proof}

Clearly, the optimality conditions of Theorem \ref{thm:final_local} are stronger than those of Theorem \ref{thm_FONC_BV}.
However, the proofs above only work on the strong assumptions that the bounds $u_a$ and $u_b$ are constant functions.
Here, it is not clear to us, under which assumptions the above techniques carry over to non-constant $u_a$ and $u_b$.

%---------------------------------------------------------------------------------
\section{Numerical tests}\label{sec5}
In this section,  the suggested algorithm is tested with selected examples. To this end,
we implemented  Algorithm \ref{alg1} in python using FEnicCS, \cite {fenicstutorial:1}.
Our examples are carried out in the optimal control setting. In particular, $f$ is given by the reduced tracking type functional
$$f(u):=\frac12\|S(u)-y_d\|_{L^2(\Omega)}^2,$$ where $S$ is the weak solution operator of some elliptic partial differential equation (PDE) specified below.
To solve the partial differential equation, the domain is divided into a regular triangular mesh, and the PDE as well as the control are discretized with piecewise linear finite elements. If not mentioned otherwise, the computations are done on a 128 by 128 grid, which results in a mesh size of $h = 0.022$.

Let us define
$j_{\epsilon,\rho}: H^{1}(\Omega)\to\R$
by
\begin{equation*}
%F_k(u):= -\Div(\psi_{\epsilon_{k}}(\nabla u))+\nabla f(u)-\lambda^a_{{k}}(u)+\lambda^b_{{k}}(u).
j_{\epsilon,\rho}(u) :=  f(u)+ \int_\Omega\psi_\epsilon(\nabla u)\dx+
\int_\Omega\frac 1{\rho}\left(M_\rho(\rho(u_a-u))+M_\rho(\rho( u-u_b))\right)\dx
\end{equation*}
with $M_\rho$ as defined in \eqref{def_Mrho}. It  is  given in our tests by the specific choice
\[
M_\rho(x):=\begin{cases}
\frac12x^2+\frac{1}{24\rho^2}\quad&\text{ if }x>\frac{1}{2\rho},\\
\frac\rho6(x+\frac1{2\rho})^3&\text{ if }|x|<\frac{1}{2\rho},\\
0&\text{otherwise.}
\end{cases}
\]
let us recall that we use the following function to approximate the $BV$-seminorm:
\[
\psi_{\epsilon}= \sqrt{\epsilon+t^2}+\epsilon t^2.
\]
	Concerning the continuation strategy for the parameters $\epsilon$ and $\rho$ in Algorithm \ref{alg1}, we set $\epsilon_0:=0.5$ in the initialization and decrease $\epsilon$ by factor $0.5$ after each iteration.
	The penalty parameter is increased by factor $2$  after every iteration and is initialized with $\rho_0 := 2$.
	Algorithm \ref{alg1} is stopped if the following termination criterion is satisfied:
	\begin{equation}\label{alg: termination}
	R^\rho_k\le10^{-4}\text{ and } R^\epsilon_{k}\le 10^{-3},
	\end{equation}
	where the residuals $R^\rho_k,\: R^\epsilon_k$ are given by
	\[
	R_k^\rho:=\|(u_a-u_k)_+\|_{L^2(\Omega)}+\|(u_k-u_b)_+\|_{L^2(\Omega)}+(\lambda_k^a,u_a-u_k)+(\lambda_k^b,u_k-u_b).
	\]
	and
	\[
	R^\epsilon_k := \|\nabla u_k\|_{L^1(\Omega)} - \langle \mu_k,\nabla u_k\rangle
	\]
	with
	$
	\mu_k:=\frac{\nabla u_k}{\sqrt{\epsilon_k+|\nabla u_k|^2}}
	$
	as in the proof of Theorem \ref{thm:final_prelim}.
	Here, the residuum $R^\rho_k$ measures the violation of the box-constraints and of the complementarity condition in Theorem \ref{thm:final}.
	Let us discuss the choice of the residuum $R^\epsilon$.
	It can be interpreted as a residual in the subgradient inequality.
	Since $\|\mu_k\|_{L^\infty(\Omega)^n} \le1$, we have $\int_\Omega \mu_k \cdot \nabla v\dx \le \|\nabla v\|_{L^1(\Omega)}$ for $v\in W^{1,1}(\Omega)$.
	Hence, $R^\epsilon_k\ge0$.
	This implies
	\[
	 \langle \mu_k, \ \nabla v - \nabla u_k\rangle \le \|\nabla v\|_{L^1(\Omega)} - \|\nabla u_k\|_{L^1(\Omega)} + R^\epsilon_k \quad \forall v\in W^{1,1}(\Omega).
	\]
	Hence, $\mu_k$ can be interpreted  is an element of the $\varepsilon$-subdifferential to the error level $\varepsilon:=R^\epsilon_k$.

\subsection{Globalized Newton Method for the subproblems}

To solve the variational subproblems of form \eqref{Peps_penal}, i.e.,
\[
\min_{u\in H^1(\Omega)} j_{\epsilon,\rho},
\]
we use a globalized Newton method.
Let us recall the notation
\begin{align*}
&\lambda^a(u):={\max}_{\rho}(\rho(u_a-u)),\\
&\lambda^b(u):={\max}_{\rho}(\rho(u-u_b)).
\end{align*}
and introduce
\[
\Lambda^a(u):=-{\max}_{\rho}'(\rho(u_a-u))=\begin{cases}
-1\quad &\text{if } \rho(u_a-u)>\frac{1}{2\rho},\\
-\rho\left(\rho(u_a-u)+\frac{1}{2\rho}\right)&\text{if } |\rho(u_a-u)|<\frac{1}{2\rho},
\\ 0&\text{otherwise},
\end{cases}
\]
and
$$\Lambda^b(u):={\max}_{\rho}'(\rho(u-u_b))=\begin{cases}
1\quad &\text{if } \rho(u-u_b)>\frac{1}{2\rho},\\
\rho\left(\rho(u-u_b)+\frac{1}{2\rho}\right)&\text{if } |\rho(u-u_b)|<\frac{1}{2\rho}, \\
0&\text{otherwise}.
\end{cases}$$
%
%To obtain global convergence of the method to a stationary point, we apply the Newton iteration with a line search strategy to solve \eqref{Peps_penal}.
The Newton method with a line search strategy is given as follows:

\begin{algorithm} [Global Newton method]\label{alg:semi_smoothN}
Set $k=0$,  $\rho>0,\:\epsilon\in(0,1)$,  $u_a,u_b\in\R$, $\eta >0,\: p>2$, $\phi\in(0,1),\tau\in(0,\frac12)$.
Choose $u_0\in H^1(\Omega)$.
\begin{enumerate}

\item Compute the search direction $w_k$ by solving
\begin{equation} \label{alg_eq_GS}
 j''(u_{k})w =-\nabla j(u_k)(u),
 \end{equation}
 where
\[
\nabla j(u):= -\Div(\psi_{\epsilon}(\nabla u))+\nabla f(u)-{\max}_{\rho}(\rho(u_a-u))+{\max}_{\rho}(\rho(u-u_b))
\]
and
\[
j''(u)d = -\Div\left(\psi_\epsilon''(\nabla u)\nabla d\right)+f''(u)d -\rho\Lambda^a(u) d+\rho\Lambda^b(u) d.
\]
\\
If  $\nabla j(u_k)\cdot w_k \le -\eta \|w_k\|^p$:  set $w_k:=-\nabla j(u_k)$.
\item (line search) Find $\sigma_k:=\max\{\phi^l:l=0,1,2,...\}$ such that
\[
j(u_k+\sigma_k w_k)-j_k(u_k)\le \tau\sigma_k \nabla J(u_k)\cdot w_{k}.
\]
\item Set $u_{k+1}:=u_k+\sigma w_k$.
\item If a suitable stopping criteria is satisfied: Stop.
\item Set $k:=k+1$ and go to step 1.

\end{enumerate}
\end{algorithm}

Let us provide details regarding the implementation of Algorithm \eqref{alg:semi_smoothN}. % which we used throughout our tests.
In the initialization of Algorithm \ref{alg:semi_smoothN}, we set $\phi = 0.5$, $\tau = 10^{-4}$, $\eta = 10^{-8}$ and $p=2.1$.
In addition, we employed the following
 termination criterion:
\[
\text{If } \|u_{k+1}-u_k\|_{L^2(\Omega)}+\|y_{k+1}-y_k\|_{L^2(\Omega)}+\|p_{k+1}-p_k\|_{L^2(\Omega)} < 10^{-10}: \quad \text{ Stop.}
\]
That is, if there is no sufficient change between consecutive iterates, we assume that the method resulted in a stationary (minimal) point of the subproblem.

%\subsection{Numerical tests}

\subsection{Example 1: linear elliptic PDE}
First, we consider the optimal control problem
\begin{equation}\label{ex1:cnvprob}
\min\limits_{u\in BV(\Omega)} \frac12\|y-y_d\|^2_{L^2(\Omega)}+\beta\BVSN{u}
\end{equation}
subject to
\[
-\Delta y = u \text{ on } \Omega,\: y=0 \text{ on } \partial\Omega
\]
and the box constraints
\[ u_a\le u(x)\le u_b \quad\text{ f.a.a. } x\in\Omega.
\]
Note that \eqref{ex1:cnvprob} as well as the subproblem
\begin{align}\label{ex1:subprob}
\min\limits_{u\in BV(\Omega)} J_{\epsilon,\rho}(y,u) :=  \frac12\|y-y_d\|^2_{L^2(\Omega)}&+\beta\int_\Omega\psi_\epsilon(\nabla u)\dx\notag\\
&+\int_\Omega\frac 1{\rho}\left(M_\rho(\rho(u_a-u))+M_\rho(\rho( u-u_b))\right)\dx\\
\text{s.t. } -\Delta y = u \text{ on } \Omega,&\: y=0 \text{ on } \partial\Omega,\notag
\end{align}
are convex and uniquely solvable.
Let us introduce the adjoint state $p \in H^1_0(\Omega)$ as the solution of the partial differential equation
\[
-\Delta p =y-y_d \text{ on } \Omega,\: y=0 \text{ on } \partial\Omega.
\]
Applying Algorithm \ref{alg:semi_smoothN} to the reduced functional of problem \eqref{ex1:subprob}
 results in the following system of equations that has to be solved in each Newton step:
\[
G(y,p,u)(\delta y,\delta p,\delta u) = F(y,p,u),
\]
where $F$ is given by
\[
F(y,p,u) := \begin{pmatrix}
-\Delta y-u \\ -\Delta p -(y-y_d) \\
p-\beta \Div(\psi'_\epsilon(\nabla u))-{\max}_{\rho}(\rho(u_a-u))+{\max}_{\rho}(\rho(u-u_b))
\end{pmatrix}.
\]
The equation $F(y,p,u)=0$
is the optimality system to problem \eqref{ex1:subprob}.
The derivative of $F$ in direction $(\delta y,\delta p,\delta u)\in H^1_0(\Omega)\times H^1_0(\Omega)\times H^1(\Omega)$ is given by
\[
G(y,p,u)(\delta y,\delta p,\delta u) = \begin{pmatrix}
-\Delta\delta y -\delta u \\ -\Delta \delta p -\delta y\\
\delta p-\beta \Div(\psi''_\epsilon(\nabla u))\nabla \delta u -\rho\Lambda^a \delta u+\rho\Lambda^b \delta u.
\end{pmatrix}.
\]
The solution of the Newton step \eqref{alg_eq_GS} is then given by $w:=\delta u$.

We adapt the example problem data from \cite{ClasonKunisch11}.
Here, $\Omega = [-1,1]^2$ and
$$y_d:= \begin{cases}
1\quad&\text{on }(-0.5,0.5)^2\\
0&\text{otherwise}
\end{cases}.$$  In the computations we set $-u_a = u_b = 10$ and $\beta = 0.0001$.
This example (without additional box constraints) was also used in \cite{hafemeyer2020}.

In Table \ref{tab:conv_rate_Ex1}, we see the convergence behavior of iterates.
Here, the errors
$$E_{u}:= \|u_{k}-u_{ref}\|_{L^2(\Omega)}, \quad E_{J}:=  |J_k-J_{ref}| $$
are presented, where $u_{ref}$ is the final iterate after the algorithm terminated at step $k=19$
and $J_{ref}:=J(u_{ref})$.
Furthermore, we
observe $R^\epsilon = O(\sqrt{\epsilon})$
 and $R^\rho = O(\frac1\rho)$ as $\epsilon \sim 2^{-k}$ and $\rho_k \sim 2^k$.

\begin{table}[htb]
	\centering
	\begin{tabular}{ccccc}
		\hline $k$ & $E_u$ &$E_J$ &$R^\epsilon_k$ & $R^\rho_k$\\
		\hline
		12 & 1.11 & $6.0\cdot 10^{-4}$  &$ 8.0 \cdot 10^{-3} $ &  $1.3\cdot 10^{-9}$ \\
		13 & 0.80 & $3.5\cdot 10^{-4}$ & $ 5.9 \cdot 10^{-3} $ & $6.7\cdot 10^{-10}$\\
		14 & 0.56 & $1.9\cdot 10^{-4}$ & $ 4.2 \cdot 10^{-3}$& $3.4\cdot 10^{-10}$ \\
		15& 0.34 & $1.0\cdot 10^{-4}$ & $ 3.0 \cdot 10^{-3} $ & $1.7\cdot 10^{-10}$ \\
		16 & 0.17 & $5.5\cdot 10^{-5}$ & $2.1 \cdot 10^{-3} $ & $8.3\cdot 10^{-11}$\\
		17 & 0.07 & $2.4\cdot 10^{-5}$ &  $ 1.5 \cdot 10^{-3} $ & $4.2\cdot 10^{-11}$ \\
		18 & 0.02 & $8.2\cdot 10^{-6}$ &  $ 1.1\cdot 10^{-3} $ & $2.1\cdot 10^{-11}$\\
		19 & --- & --- & $ 7.6 \cdot 10^{-4} $ &$1.1\cdot 10^{-11}$ \\
		\hline
		\end{tabular}
\caption{Computed errors during the final iterations.}
\label{tab:conv_rate_Ex1}
\end{table}

Figure \ref{fig:u_lin} shows the optimal control. The result is in agreement with the results obtained in \cite{hafemeyer2020}.
In Figure \ref{fig:lin_ref} the computed optimal controls are depicted for the unconstrained case (left), i.e., constraints are inactive during the computation process.
The right plot shows the optimal control $u$, when lower and upper bound are set to $u_a =-5$ and $u_b= 18$.

\begin{figure}
	\centering
	\begin{subfigure}[b]{0.4\textwidth}
		\centering
		\includegraphics[width=\textwidth]{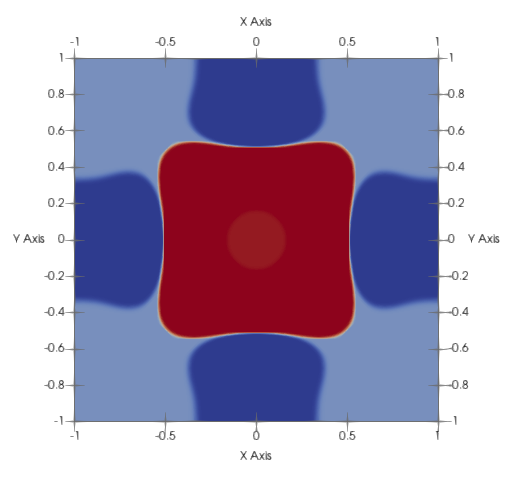}
	\end{subfigure}
\hfill
	\begin{subfigure}[b]{0.5\textwidth}
		\centering
		\includegraphics[width=\textwidth]{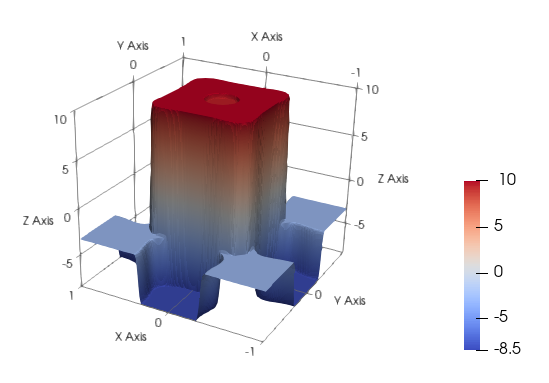}
	\end{subfigure}
	\caption{Optimal control $u$.}%
	\label{fig:u_lin}
\end{figure}

\begin{figure}
	\centering
	\begin{subfigure}[b]{0.4\textwidth}
		\centering
		\includegraphics[width=\textwidth]{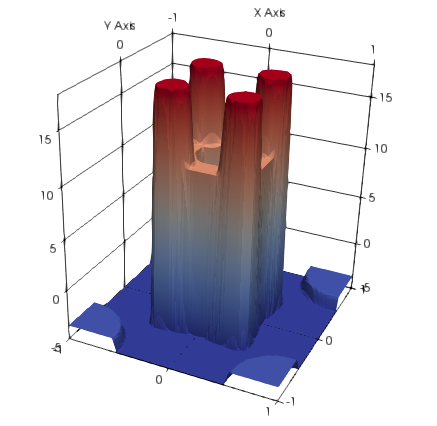}
	\end{subfigure}
\hfill
	\begin{subfigure}[b]{0.5\textwidth}
		\centering
		\includegraphics[width=\textwidth]{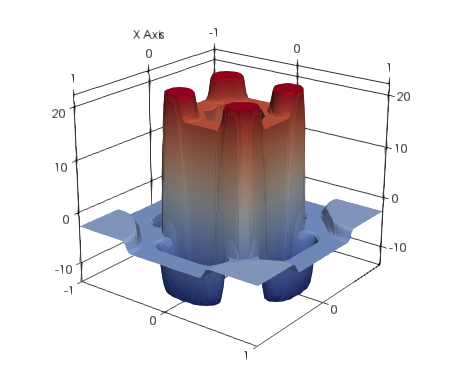}
	\end{subfigure}
	\caption{Optimal control $u$ for different choices of $u_a,u_b$.}%
	\label{fig:lin_ref}
\end{figure}

\subsection{Example 2: Semilinear elliptic optimal control problem}
Let us now consider the following problem with semilinear state equation.
That is, we study the minimization problem
\[
\min_{u\in U_{ad}}f_{sl}(u)+\beta\BVSN{u},
\]
where $f_{sl}$ is given by the standard tracking type functional
$u\mapsto\|y_u-y_d\|^2_{L^2(\Omega)}$,
and $y_u$ is the weak solution of the semilinear elliptic state equation
\[
-\Delta y+  y^3=u \quad\text{in }\Omega,\quad y = 0\quad\text{on }\partial\Omega.
\]
%The non-linearity in the state equation makes the problem nonconvex.
The adjoint state $p\in H^1_0$ is given now as solution to the equation
\[
-\Delta p+3y^2p= y-y_d \quad\text{in }\Omega,\quad p = 0\quad\text{on }\partial\Omega.
\]
For this example, system of equations \eqref{alg_eq_GS} for the state $y\in H^1_0(\Omega)$, the adjoint state $p\in H^1_0(\Omega)$ and the control variable $u\in H^1(\Omega)$ is given by
\[
 G(y,p,u)(\delta y,\delta p,\delta u) = F(y,p,u)
\]
with
\[
F(y,p,u) := \begin{pmatrix}
-\Delta y+y^3-u \\
-\Delta p +3y^2p-(y-y_d) \\
p-\beta \Div(\psi'_\epsilon(\nabla u))-{\max}_{\rho}(\rho(u_a-u))+{\max}_{\rho}(\rho(u-u_b))
\end{pmatrix}.
\]
The derivative in direction $(\delta y,\delta p,\delta u)\in H^1_0(\Omega)\times H^1_0(\Omega)\times H^1(\Omega)$ is given with
\[
G(y,p,u)(\delta y,\delta p,\delta u) = \begin{pmatrix}
-\Delta\delta y +3y^2\delta y-\delta u \\
 -\Delta \delta p +3y^2\delta p-\delta y+6yp\delta y\\
\delta p -\beta \Div(\psi''_\epsilon(\nabla u))\nabla \delta u -\rho\Lambda^a \delta u+\rho\Lambda^b \delta u.
\end{pmatrix}.
\]
%The data is given by $\beta =0.0001$, $b = 12$ and $y_d = 4\sin(2\pi x_1)\sin(\pi x_2)e^{x_1}$.
%Furthermore, we set $\rho_0 = 10$ and $\epsilon_0=0.9$.

The data is given as in Example 1.
The optimal control is depicted in Figure \ref{fig:u_semlin}. It is close to the solution of Example 1.
Let us consider the performance of the algorithm on different levels of discretization for this example.
Table \ref{table:Ex2_iterations} shows the number of outer iterations ($\sharp$it), as well as the total number of newton iterations ($\sharp$newt) needed until the stopping criterion \eqref{alg: termination} holds for increasing meshsizes.
%We observe that more Newton iterations are needed for finer grids.
The last column shows the final objective value
 $ J_{\epsilon,\rho}$.
The residuals $R^\epsilon$ and $R^\rho $ behaved as in Example 1.

\begin{table}[htbp]
	\centering
	\begin{tabular}{cccccc}
		\hline
		$h$  & $\sharp$it & $\sharp$newt &$\epsilon_{final }$ & $\rho_{final}$ & $J_{\epsilon,\rho}$ \\
		\hline
\rule{0pt}{1\normalbaselineskip}
		0.088  & 16 & 182 & $ 2^{-16}$& $2^{16}$ & 0.0596 \\
			0.044 &19& 201 & $2^{-19} $& $2^{19}$ & 0.0685 \\

		0.022  & 19& 314   & $2^{-19} $ &$2^{19}$  & 0.0737 \\

		0.011 &19& 486 & $2^{-19} $ & $2^{19}$	& 0.0767\\

		\hline
	\end{tabular}
	\caption{Number of iterations and newton steps for different mesh-sizes.}
	\label{table:Ex2_iterations}
\end{table}

\begin{figure}
	\centering

		\includegraphics[height=6cm]{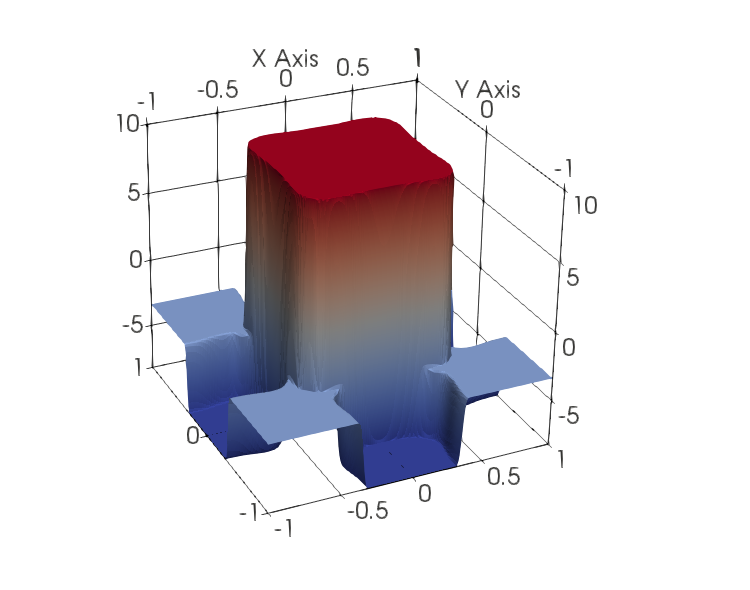}

	\caption{Optimal control $u$ for the semilinear problem.}%
	\label{fig:u_semlin}
\end{figure}

\subsection{Experiments with non-constant constraints}

So far our analysis and numerical experiments are restricted to the case where  $u_a,u_b$ are constant functions.
This assumption was needed to show the boundedness of multipliers $\lambda_k^a(u),\lambda_k^b(u)$ in $L^2(\Omega)$ in Lemma \ref{lem_lambda_bounded},
which is crucial for the final result Theorem \ref{thm:final}.
For this section we tested Algorithm \ref{alg1} also for non-constant functions $u_a,u_b\in L^\infty(\Omega)$.

Here, we consider again the linear optimal control problem and data from Example 1 with different choices for $u_a,u_b$:
	\begin{align}\label{ex1:ub_sin}
(i)\quad	&u_a:= -100, \: u_b(x_1,x_2):=8\sin(\pi x_1)\sin(\pi x_2),\\
(ii)\quad	\label{ex1:ub_cont2}
	&u_a:= -100, \: u_b(x_1,x_2):= -4(x_1-0.5)^2-4x_2^2+10,
	\end{align}

In  Figure \ref{fig_L2_discretization}  the behavior of the quantity
$\|\lambda_k^a(u)\|^2_{L^2(\Omega)}+\|\lambda_k^b(u)\|^2_{L^2(\Omega)}$ is plotted along the iterations, i.e., for increasing $\rho_k$, for different discretization levels.
In Figure \ref{fig:u_nc}, the respective solution plots are shown.
While the multipliers seem to be bounded for one example, their norm grows with $\rho$ (and thus with $\epsilon$) for the other example.
Clearly, more research has to be done to develop necessary and sufficient conditions for the boundedness of the multipliers.

\begin{figure}
	\centering
	\begin{subfigure}[b]{0.45\textwidth}
		\centering
		\includegraphics[width=\textwidth]{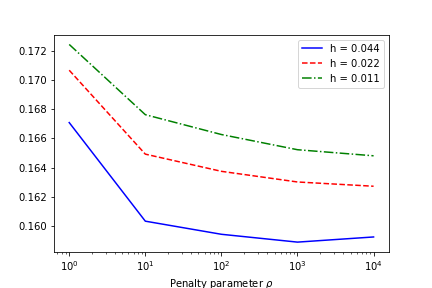}
	\end{subfigure}
	\hfill
	\begin{subfigure}[b]{0.45\textwidth}
		\centering
		\includegraphics[width=\textwidth]{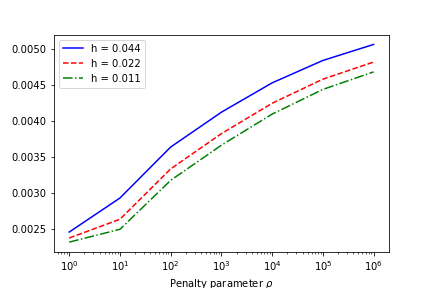}
	\end{subfigure}
	\caption{The $L^2$-norm of multipliers $\lambda^a_k(u)$, $\lambda_k^b(u)$ for Example \eqref{ex1:ub_sin} (left) and \eqref{ex1:ub_cont2} (right).}%
	\label{fig_L2_discretization}
\end{figure}

\begin{figure}
	\centering
	\begin{subfigure}[b]{0.45\textwidth}
		\centering
		\includegraphics[height = 5cm]{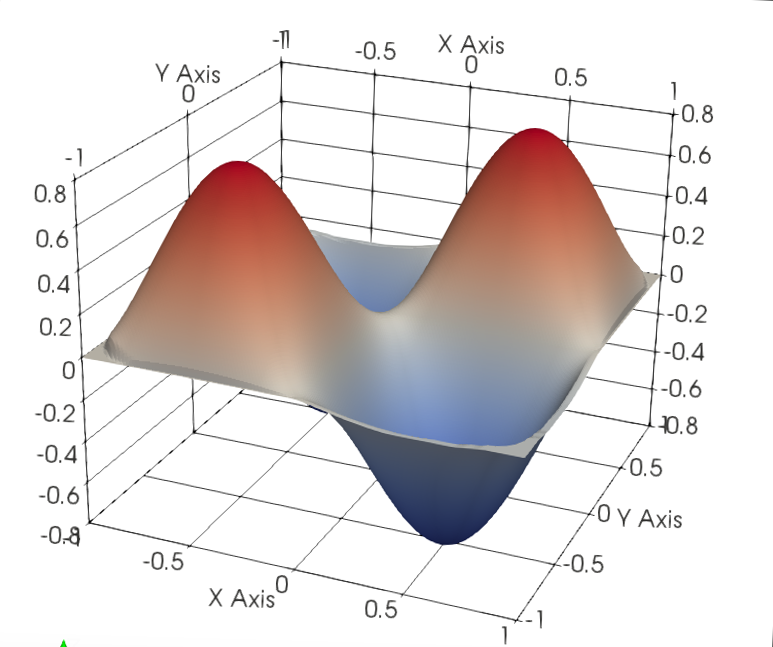}
	\end{subfigure}
	\hfill
	\begin{subfigure}[b]{0.45\textwidth}
		\centering
		\includegraphics[height= 5cm]{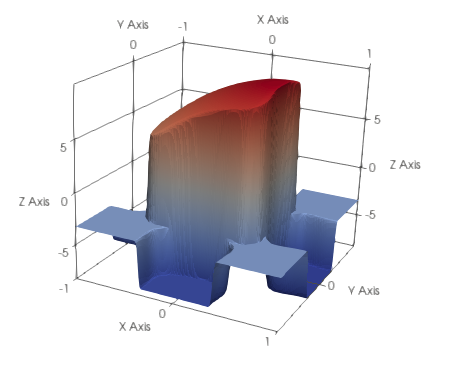}
	\end{subfigure}
	\caption{Optimal control $u$ for Example \eqref{ex1:ub_sin} (left) and \eqref{ex1:ub_cont2} (right).}%
	\label{fig:u_nc}
\end{figure}

\section*{Acknowledgement}
The authors are grateful to Gerd Wachsmuth for an inspiring discussion that led to an improvement of Theorem \ref{thm:final_prelim} and subsequent results.

% \clearpage
%
% % \newpage
\bibliography{literature}
\bibliographystyle{plain_abbrv}

%\printbibliography

\end{document}